\documentclass[a4paper,10pt]{article}
\usepackage[utf8]{inputenc}
\usepackage{carlocomm,geometry}

\usepackage{lmodern}
\title{Implicit Commitment in a General Setting}
\author{Mateusz \L e\l yk and Carlo Nicolai}
\date{}

\begin{document}

\maketitle


\begin{abstract}
G\"odel’s Incompleteness Theorems suggest that no single formal system can capture the
entirety of one’s mathematical beliefs, while pointing at a hierarchy of systems of increasing
logical strength that make progressively more explicit those implicit assumptions. This
notion of implicit commitment motivates directly or indirectly several research programmes
in logic and the foundations of mathematics; yet there hasn’t been a direct logical analysis of
the notion of implicit commitment itself. In a recent paper, we carried out an
initial assessment of this project by studying necessary conditions for implicit commitments;
from seemingly weak assumptions on implicit commitments of an arithmetical system $S$,
it can be derived that a uniform reflection principle for $S$ -- stating that all numerical
instances of theorems of $S$ are true -- must be contained in $S$’s implicit commitments. This
study gave rise to unexplored research avenues and open questions. This paper addresses
the main ones. We generalize this basic framework for implicit commitments along two
dimensions: in terms of iterations of the basic implicit commitment operator, and via a study
of implicit commitments of theories in arbitrary first-order languages, not only couched in
an arithmetical language.
\end{abstract}

\section{Introduction}

As forcefully argued by Solomon Feferman, G\"odel's Incompleteness Theorems suggest that no single formal system can capture the entirety of one's mathematical beliefs, while pointing at a hierarchy of systems of increasing logical strength that make progressively more explicit those \emph{implicit} assumptions.\footnote{See for instance, \cite[p.1]{fef91}.} This notion of \emph{implicit commitment} motivates directly or indirectly several research programmes in logic and the foundations of mathematics. To mention just a few, Turing's ordinal logics \cite{tur39} anticipated  Feferman's foundations of predicative mathematics given the natural numbers -- i.e.~charaterized as the portion of mathematical thought implicit in the acceptance of the natural numbers --  via iterations of recursive systems along autonomous ordinals \cite{fef62,fef64}. Feferman later tried to simplify the presentation of these predicatively acceptable ordinals via a notion of truth \cite{fef91}, and via his framework for explicit mathematics.\footnote{A comprehensive bibliography on explicit mathematics can be found at  https://home.inf.unibe.ch//$\sim$til/em{\textunderscore}bibliography/.} The study of reflection principles and theories of truth led to a proliferation of studies in proof-theoretic ordinal analysis \cite{sch79,fra04,bek15,bepa19}, and techniques in theories of truth \cite{jal99,fuj12,hole17,finiho17,cie17}.\footnote{More philosophical works have recently tackled the notion of implicit commitment directly. \cite{dea14} and \cite{nipi18} are some examples.}

Despite this interest in frameworks that are motivated by implicit commitments, there hasn't been a direct logical analysis of the notion of implicit commitment itself. \cite{leni22} carried out an initial assessment of this project by studying minimal formal components of implicit commitments; from seemingly weak assumptions on implicit commitments of an arithmetical system $S$, it can be derived that a uniform reflection principle for $S$ -- stating that all numerical instances of theorems of $S$ are true -- must be contained in $S$'s implicit commitments. These assumptions are (i) that (provably) logically equivalent theories have identical commitments (\textsc{invariance}),\footnote{More precisely,in \cite{leni22} the proof-theoretical equivalence in question amounts proof-transformations witnessed by elementary functions. We weaken this assumption to p-time reducibility. } and (ii) that if, provably in a very weak metatheory, all numeral instances $A$ are axioms of $S$, then $\forall x A$ is in the commitments of $S$ (\textsc{reflection)}.  This work is reviewed in Section  \ref{sec:basic}. 

The framework in \cite{leni22} gave rise to unexplored research avenues and open questions. This paper addresses the main ones. It focuses on generalizations of the basic framework for implicit commitments along two dimensions: in terms of iterations of the basic implicit commitment operator (Section \ref{sec:iterat}), and via a study of implicit commitments of theories in arbitrary first-order languages, not only couched in an arithmetical language. We study iterations of the basic operator $\mathcal{J}$ along an elementary presentation of an ordinal number, which permits us to uniformly justify the transfinite iterations of uniform reflection (Section \ref{sec:iterat}). To carry out the second kind of generalization, we study different options; after dismissing a na\"ive approach based on the addition of ``domain constants'' (Section \ref{sec:domcon}), we discuss several alternative generalizations of \textsc{invariance} based on strengthenings of mutual relative interpretability. Our preferred generalization is inspired by the principle that bi-interpretable theories have bi-interpretable commitments.\footnote{Bi-interpretability is a strong notion of mutual reduction, much stricter than mutual interpretability; it will be introduced in Section \ref{sec:thre}.} After combining this principle with a suitable generalization of \textsc{reflection} , we show that the main properties of the basic framework carry over to generalized implicit commitments (Section \ref{sec:uniref}): given a theory $S$, a suitable form of uniform reflection for $S$ is bound to be in the implicit commitments of $S$, and that this form of uniform reflection amounts to a natural interpretation of the generalized implicit commitment operator.

\section{Technical Preliminaries}\label{sec:prelim}

\subsection{Theories and Coding} 
Our main notion of proof-theoretic reduction will be the notion of \emph{p-time reducibility}. We assume familiarity with Buss' theory $\sot$, as given for instance in \cite{bus86} and \cite{hapu98}. We call $\lnat$ the standard signature of arithmetic extended with function symbols to develop $\sot$. The formalization of syntactic notions and operations for a standard formal system in $\sot$ is based on the fact that such notions can be coded by p-time functions and predicates. The functions that are provably total in $\sot$ are in fact precisely the p-time computable functions.\footnote{In effect, $\sot$ stands to the p-time functions as $\mathtt{I}\Sigma_1$ stands to the primitive recursive functions.}

We provide an informal development of such a coding to justify the choice of $\sot$. It suffices to code strings of symbols coming from a finite alphabet $\{a,b\}$. Strings are arranged in length-first, alphabetic order (or shortlex order); assuming the empty string $\epsilon$, one first lists strings of length $0$, then of length $1$, and so on, where strings of length $n$ are ordered alphabetically. The code of a string is the ordinal number of a string in such a list. There are $2^n$ strings of length $n$; therefore, the number of strings of length $\leq n$ is
\[
    2^n+2^{n-1}+\ldots+ 1= 2^{n+1}-1. 
\]
The code $s$ of a string of length $n$ will be $2^{n+1}-1$, and therefore it is $\mc{O}(2^n)$ big. Symmetrically, for $s$ coding a string, its length $|s|$ is $\mc{O}(\mathtt{log}_2 s)$. We estimate the growth on codes of the fundamental operations of \emph{concatenation} and \emph{substitution}. Concatenation of codes $s_1\smf s_2$ has the same growth rate as multiplication: since $|s_1\smf s_2|=|s_1|+|s_2|$, the code $s_1\smf s_2$ will be $\mc{O}(2^{|s_1|+|s_2|})$-big, so it is estimated to be 
\[
\mc{O}(2^{|s_1|+|s_2|})=\mc{O}(s_1\times s_2).
\]
For substitution of codes $s(t/x)$,\footnote{Here we are assuming that the expression  $s(t/x)$ stands for the code of the result of substituting, in the string coded by $s$, all occurrences of the string coded by $x$ with the string coded by $t$.} the worst case scenario is that $s$ ``amounts to'' $|s|$-many $x$s. In such a scenario, $|s(t/x)|=|t|\times |s|$ and therefore $s(t/v)$ is estimated to be
\[
    \mc{O}(2^{|s|\times |t|})=\mc{O}(s^{|t|}). 
\]
This logarithmic growth on codes is precisely what p-time recursion, and therefore $\sot$, can accommodate. An important caveat concerns the formalization of the function that sends a number to the code of its corresponding numeral. If numerals are formalized as 
\[
    \underbrace{\mathtt{S}\ldots\mathtt{S}}_{\text{$n$-many}}\ovl{0},
\]
then the code of the numeral for $n$ will be $\mc{O}(2^n)$-big, and therefore exponential. Therefore, \emph{dyadic numerals} are adopted:
\begin{align*}
    &\ovl{0}=0&& 
    &&\ovl{2n+1}=\mathtt{S}((\ovl{1}+\ovl{1})\times \ovl{n})&& \ovl{2n+2}=(\ovl{1}+\ovl{1})\times \ovl{n+1}.
\end{align*}

In bounded arithmetical theories such as $\sot$ a central role is played by the \emph{bounded hierarchy} for $\Sigma_n^b$-formulae, which parallels the arithmetical hierarchy except that at the bottom of the scale $\Sigma_0^b$ one has formulae containing only quantification bounded by terms of form $|t|$. The hierarchy then keeps track of alternating quantifiers bounded by ordinary terms.\footnote{See for instance \cite[\S V, Def.~4.2]{hapu98}.} 

The bounded hierarchy is used to define our notion of a \emph{theory}; in what follows a theory is a  $\Delta_0^b$-formula $\tau(x)$ such that 
\[
    \sot \vdash \forall x (\tau(x)\ra \mathtt{Sent}_{\lnat}(x)).
\]
A theory $\tau$ is \emph{schematic} if there is a first order formula $\vphi(P)$ with a free second order variable $P$
such that $\tau(x)$ says:
\begin{quote}
There is a formula $\psi(y)$ such that $x$ is the result of replacing $P(y)$ with $\psi(y)$ in $\vphi$.\footnote{It is implicit in the definition that $\psi(y)$ may contain additional free variables (parameters) and that the substitution in $\vphi$ is well-defined in the sense that none of the parameters is captured by $\vphi$-quantifiers.}
\end{quote}
It's important to notice that, according to this definition,  Reflection Principles such as $\mathtt{RFN}(\tau)$ -- see below for a definition -- do not conform to it , because the formula $\vphi(v)$ over which reflection is performed is both used and
mentioned (via the Gödel code of $\vphi(v)$) in its instances.

Finally, $\sot$ is finitely axiomatizable and sequential in the sense of \cite{pud85}. The finite axiomatizability of $\sot$ differentiates it from other analogous theories that are able to cope with functions of polynomial growth rate such as $\mathtt{I}\Delta_0+\Omega_1$, where $\Omega_1$ is the totality of the function $x^{|x|}$; It is not known whether $\mathtt{I}\Delta_0+\Omega_1$ is finitely axiomatized.

\subsection{Inter-theoretic Reductions.}\label{sec:thre}

A prominent role in the paper will be played by \emph{p-time reducibility}.  
Let $\tau$ and $\tau'$ be formulated in the same signature $\Sigma$. We say that $\tau$ is p-time-reducible to $\tau'$ -- in symbols, $\tau \leq_{\mathtt{pt}}\tau'$ -- if there is a p-time function $f$ such that 
 \[
    \sot \vdash (\forall \vphi \in\mc{L}_\Sigma)( \mathtt{Proof}_{\tau}(y,\vphi)\ra \mathtt{Proof}_{\tau'}(f(y),\vphi)).
 \]

 To extend the results in \cite{leni22}, we would like to compare the implicit commitments of theories formulated in different signatures. We will employ the notion of a \emph{relative interpretation}. 
%
 Let $\Sigma$ and $\Xi$ be one-sorted, first-order, finite, relational signatures. A (one-sorted, one-dimensional, parameter-free) relative translation $\tau\colon \Sigma \to \Xi$ can be seen as a pair $(\delta,F)$, where $\delta$ is a unary formula of $\mc{L}_\Xi$, and $F$ a function that sends $n$-ary relation symbols in $\Sigma$ to $\mc{L}_\Xi$-formulae with $n$ free variables.\footnote{It is assumed a machinery to rename variables to avoid clashes.} The translation $\tau$ then commutes with propositional connectives and relativizes quantifiers to $\delta$:
 \begin{align*}
    (R(x_1,\ldots, x_n))^\tau &:\lra F(R)(x_1,\ldots,x_n)\\
    (\neg A)^\tau &:\lra \neg A^\tau\\
    (A\land B)^\tau &:\lra A^\tau \land B^\tau\\
    (\forall x A)^\tau &:\lra \forall x(\delta(x)\ra A^\tau)
 \end{align*}
We occasionally abbreviate $\forall x(\delta_N(x)\ra \ldots $ with $\forall x\coltyp N \ldots$\,. Let $T$ and $W$ be $\Sigma$ and $\Xi$ theories, respectively. The relative translation $\tau$ gives rise to a \emph{relative interpretation} whenever 
\[
\text{$T\vdash A$ only if $W\vdash A^\tau$}.
\] We say that $\tau$ is a faithful interpretation of $T$ in $W$ if the stronger claim
\[
\text{$T\vdash A$ if and only if $W\vdash A^\tau$}
\]
obtains.
We will not distinguish between translations and the interpretations supported by them in what follows.

We will consider two prominent notions of equivalence of theories resulting from corresponding notions of equivalence between interpretations. Two interpretations $\tau,\sigma\colon T\to W$, are \emph{identical} iff:
\begin{align*}
    & W \vdash \delta_\tau(x)\lra \delta_\sigma(x)\\
    &W\vdash \bigwedge_{1\leq i \leq n}\delta_\tau(x_i)\ra (R^\tau(x_1,\ldots,x_n)\lra R^\sigma(x_1,\ldots,x_n))&&\text{for any $R\in \mc{L}_T$}.
\end{align*}
$T$ and $W$ are \emph{definitionally equivalent} if there are $\tau\colon T\to W$ and $\xi \colon W \to T$ such that $T$ proves that $\xi \circ \tau$ is identical to the identity interpretation on $T$, and $W$ proves that $\tau \circ \xi$ and the identity on $W$ are identical. 

As to the second notion of equivalence of interpretations, given two interpretations $\tau,\sigma\colon T\to W$, a $W$-definable \emph{isomorphism} between $\tau$ and $\sigma$ is a $\mc{L}_W$-formula $I(x,y)$ such that $W$ proves the following:
\begin{align}
    & I(x,y)\ra \delta_\tau(x)\land \delta_\sigma(y)\\
    & \forall x (\delta_\tau(x) \ra \exists y(\delta_\sigma(y)\land I(x,y)))\\
    & \forall y(\delta_\sigma(y) \ra \exists x (\delta_\tau(x)\land I(x,y)))\\
    & I(x,y)\land x=^\tau u \land y=^\sigma v \ra I(u,v)\\
    & I(x,y)\land I(x,v)\ra y=^\sigma v\\
    &I(x,y)\land I(u,y)\ra x=^\tau u\\
    & \bigwedge_{1\leq i\leq n} I(x_i,y_i)\ra ( R^\tau (x_1,\ldots,x_n)\lra R^\sigma(y_1,\ldots,y_n))&&\text{for any $R\in \mc{L}_T$}
\end{align}
The theories $W$ and $T$ are \emph{bi-interpretable} if there are interpretations $\tau \colon T\to W$ and $\sigma\colon W \to T$ such that $T$ proves that there is a $T$-definable isomorphism between $\sigma\circ \tau$ and the identity interpretation on $T$, and $W$ proves that there is a $W$-definable isomorphism between the identity on $W$ and $\tau\circ \sigma$. 

It will be useful in what follows to separate the two conditions at work in the definition of bi-interpretability. We say that $T$ is a \emph{retract} of $W$ if there are interpretations $\tau \colon T\to W$ and $\sigma\colon W \to T$ such that $T$ proves that there is a $T$-definable isomorphism between $\sigma\circ \tau$ and the identity interpretation on $T$.

\subsection{Proof-Theoretic Reflection Principles.} We recall some basic facts concerning formal provability and soundness for the logical systems we are interested in.  Given a canonical, $\exists\Delta_1^b$-provability predicate 
\[
\mathtt{Prov}_\tau(x):\lra \exists p\, \mathtt{Proof}(p,x),\] a \emph{consistency statement} for $\tau$ is the $\forall \Delta_1^b$-formula
\[
    \mathtt{Con}(\tau):\lra \neg \mathtt{Prov}_\tau(\corn{0=1}).
\]
It can be shown in $\sot$ that $\mathtt{Con}(\tau)$ is  equivalent to local and uniform reflection restricted to $\forall \Delta_1^b$-formulae.\footnote{Where $\forall\Delta_1^b$ is a class of formulae provably equivalent to a list of unbounded universal quantifiers in front of a $\Delta_1^b$-formula.} A restricted consistency statement $\mathtt{Con}(\ovl{n},\tau)$ rules out proofs of $0=1$ smaller than $n$. 

The \emph{local reflection principle} is the collection
\[
    \tag{$\mathtt{Rfn}(\tau)$} \{ \mathtt{Rfn}(\corn{\vphi})\ra \vphi\sth \vphi\in \mathtt{Sent}_{\mc{L}_\tau}\}
\]
For arithmetical $\tau$, the \emph{Uniform Reflection Principle} is the collection 
\[
    \tag{$\mathtt{RFN}(\tau)$} \{\forall x (\mathtt{Prov}_\tau(\corn{\vphi(\dot{x})})\ra \vphi(x))\sth \vphi \in \mathtt{Form}_{\mc{L}_\tau}\},
\]
where the expression $\corn{\vphi(\dot{x})}$ stands for the result of formally substituting, in $\corn{\vphi(v)}$, the variable $v$ for the dyadic numeral for $x$.
Clearly, all members of  $\mathtt{Rfn}(\tau)$ are also members of $\mathtt{RFN}(\tau)$. It is well-known that, over a weak metatheory such as $\ea$, $\mathtt{RFN}(\tau)$ is properly stronger than $\mathtt{Rfn}(\tau)$.\footnote{See for instance, \cite{bek05}.}

Since $\mathtt{RFN}(\tau)$ can be elementarily presented in a natural way, this opens up a possibility of iterating the process of adding uniform reflection. Formally, for an arbitrary $\tau$, we define

\begin{align*}
    \mathtt{RFN}^0(\tau)&:= \tau\\
    \mathtt{RFN}^{\alpha+1}(\tau)&:=\tau+ \mathtt{RFN}(\mathtt{RFN}^{\alpha}(\tau))\\
    \mathtt{RFN}^{\alpha}(\tau)&:= \bigcup_{\beta<\alpha}\mathtt{RFN}^{\beta}(\tau). 
\end{align*}
\cite[Section 5]{bek18} describes the standard way of formalizing this progression in theories extending $\ea$.

\subsection{{Axiomatic Theories of Truth}.}\label{sec:atot} We will refer to axiomatic theories of truth in some key examples below. These are logical systems that, due to the undefinability of truth, are formulated in a language $\lnat \cup \{\T\}$, for $\T$ a unary predicate.\footnote{For an overview of the systems, see \cite{cie17} and \cite{hal14}.} 

We assume some $N\colon \sot \to \tau$. The theory $\mathtt{UTB}^-[\tau]$ is obtained by relativizing coding and syntax to $N$, and by extending $\tau$ in $\lt$ with the schema
    \[
    \tag{$\mathtt{UTB}$} (\forall x:N)(\T\corn{A(\dot{x})}\lra A(x)).
    \]

    The theory $\mathtt{SC}[\tau]$ extends $\mathtt{UTB}^-[\tau]$ with the \emph{schema}
    \[
        \tag{$\tau \subseteq \T$} (\forall x:N)(\tau(\corn{\vphi(\dot{x})})\ra \T\corn{\vphi(\dot{x})})
    \]
    for \emph{each} axiom $\vphi$ of $\tau$. We will employ the following:
    \begin{lemma}\label{lem:sccons}
        For schematic $\tau$, $\mathtt{SC}[\tau]$ is conservative over $\tau$.
    \end{lemma}

    \begin{proof}
        The proof goes back to Tarski's original argument for the definability of truth predicates for finitely many sentences. 
        
        Let $\psi_1(v)\ldots \psi_n(v)$ be the finitely many formulae occurring in the instances of $\mathtt{UTB}$ in a $\mathtt{SC}[\tau]$ proof $\mc{D}$, together with the finitely many formulae instantiating $(\tau \subseteq  \T)$ in $\mc{D}$. Consider the $\lnat$-formula
        \[
            \mf{T}(x):\lra \exists y (x=\corn{\psi_1(\dot{y})}\land \psi_1(y))\vee\ldots\vee \exists y(\corn{\psi_n(\dot{y})}=x \land \psi_n(y))
        \]
        One simply replaces $\T$ with $\mf{T}$ in $\mc{D}$. All instances of $\mathtt{UTB}$ in $\mc{D}$ are then provable in $\tau$. It can be verified that, for the finitely instances of $(\tau \subseteq \T)$ in $\mc{D}$,
        \[
            \forall x (\tau(\corn{\vphi(\dot{x})})\ra \mf{T}\corn{\vphi(\dot{x})})).
        \]
        We notice that the assumption that $\tau$ is schematic is essential. 
    \end{proof}

The axiomatic theory of truth $\mathtt{CT}^-[\tau]$,  is obtained by extending $\tau$ with the axioms\footnote{In the axioms, we write $\forall \vphi \ldots$ as an abbreviation for $\forall x (\mathtt{Sent}_{\lt}(x) \ra \ldots$.} 
\begin{align}
    \tag{$\mathtt{CT}\!P$} &\forall x_1\ldots x_n(\T\corn{P(\dot{x}_1\ldots \dot{x}_n)}\lra P(x_1,\ldots,x_n))&& \text{for $P$ in $\mc{L}_\tau$}\\[1ex]
    \tag{$\mathtt{CT}\neg$}& \forall \vphi\in \mc{L}_\tau(\T(\neg \vphi)\lra \neg \T \vphi)\\[1ex]
    \tag{$\mathtt{CT}\land$}& \forall \vphi,\psi\in \mc{L}_\tau(\T(\vphi\land \psi)\lra (\T\vphi\land \T\psi))\\[1ex]
    \tag{$\mathtt{CT}\forall$}& \forall \vphi(v) \in \mc{L}_\tau(\T(\forall v\vphi)\lra \forall x\,\T\vphi(\dot{x}/v))
\end{align}
For $\tau=\pa$, $\mathtt{CT}$ is the extension of $\mathtt{CT}^-[\pa]$ with full induction for $\lt$.

\section{Basic Principles for Implicit Commitment}\label{sec:basic}

 \cite{leni22} introduced a formal framework to capture necessary conditions for the implicit commitments of a sufficiently strong formal mathematical theory. Crucially, they consider only theories \emph{formulated in the language of arithmetic $\lnat$}. In this sub-section we summarize their main results. They axiomatize an operator $\mc{I}$ on theories, governed by the principles introduced in the following definition.
\begin{definition}[Invariance, Reflection]
Theories are taken to be $\Delta_0^b$-formulae.  Given $\sigma$ in the language $\lnat$, the collection of its implicit commitments $\mc{I}(\sigma)$ is constrained by the following principles of \emph{invariance} and \emph{reflection}: for $\tau,\tau'$ theories in $\lnat$, 
    \begin{align}
        \tag{\textsc{inv}} &\tau \leq_{\mathtt{pt}} \tau'\text{ only if }\mc{I}(\tau)\subseteq \mc{I}(\tau')\\[1ex]
        \tag{\textsc{ref}} &\sot \vdash \forall x\, \tau (\corn{\vphi(\dot{x})})\text{ only if } \forall x \vphi \in \mc{I}(\tau)
    \end{align}
\end{definition}

Given a theory $\tau$, each of \textsc{(inv)} and \textsc{(ref)} over $\tau$ can be reduced to $\tau$ by assigning a more definite  meaning of $\mc{I}$. For \textsc{(inv)}, one can consider the trivial interpretation in which $\mc{I}(\tau)=\tau$: the arithmetical soundness of $\sot$ then guarantees that $\tau\subseteq \tau'$. For \textsc{(ref)}, one considers the set
\[
    S=\{ \forall x A \sth \sot \vdash \forall x \,\tau(\corn{A(\dot{x})}) \}
\]
for $\tau$ schematically axiomatized, with $\tau$ an extension of Kalmar's Elementary Arithmetic $\ea$. Therefore, the theory $\mathtt{SC}[\tau]$ includes all members of $S$. By Lemma \ref{lem:sccons}, $S$ is conservative over $\tau$.

\begin{proposition}\label{prop:basicleni}
    For $\tau\supseteq \sot$ a theory in $\lnat$, $\mathtt{RFN}(\tau)\subseteq \mc{I}(\tau)$. 
\end{proposition}
\begin{proof}[Proof Sketch]
    The proof is given in detail in \cite{leni22} and rests on the (well-known) provability of the so-called \emph{small reflection principle} in $\tau$.  

    One first ``re-axiomatizes'' $\tau$ as
    \[
        \tau'(x):\lra \tau(x)\vee \exists y\leq x \,x= \corn{\mathtt{Proof}_\tau(y_1,\vphi(y_2))\ra \vphi(y_2)}.\footnote{Here $y$ is intended to be a pair $(y_1,y_2)$.}
    \]
    Notice that this ``re-axiomatization'' is dependent on the given $\vphi$. 
    
    By \textsc{(ref)} applied to $\tau'$, since 
    \beq\label{eq:smaref}
        \sot \vdash \forall y\, \tau'(\corn{\mathtt{Proof}_\tau(y_1,\vphi(y_2))\ra \vphi(y_2)}),
    \eeq
    we obtain that the instance of $\mathtt{RFN}(\tau)$ for $\vphi$ is in $\mc{I}(\tau')$. Since $\vphi$ is arbitrary, we in fact showed that all instances of $\mathtt{RFN}(\tau)$ are in $\mc{I}(\tau')$.
    
    Moreover, since 
    \beq
        \sot \vdash \forall y\, \mathtt{Prov}_\tau(\corn{\mathtt{Proof}_\tau(y_1,\vphi(y_2))\ra \vphi(y_2)}),
    \eeq
    we have that $\tau' \leq_{\mathtt{pt}} \tau$. Therefore, $\mc{I}(\tau')\subseteq \mc{I}(\tau)$ and therefore $\mathtt{RFN}(\tau)\subseteq \mc{I}(\tau)$. 
\end{proof}

In fact, if one accepts $\mc{I}(\tau)$ as necessary conditions for one's implicit commitments given the acceptance of $\tau$, it can be shown that a lower bound for one's implicit commitments coincides with  the acceptance of all instances of uniform reflection for $\tau$. 
\begin{proposition}
    Let 
    \[
        \mc{I}_{\mathtt{RFN}}(\tau)=\{ \vphi \sth \tau+\mathtt{RFN}(\tau)\vdash \vphi\}.
    \]  
    Then $\mc{I}_{\mathtt{RFN}}(\tau)$ satisfies (i) and (ii). 
\end{proposition}

\section{Iterations}\label{sec:iterat}

In this section we experiment with the intuitive idea that the implicit commitments of a theory should give rise to new implicit commitments: if, upon accepting an \emph{arbitrary} theory $\tau$, one is committed to the acceptance of a theory $\tau'$, then we can apply the same reasoning starting from $\tau'$ and so on. Moreover, if one sees the legitimacy of the whole process of passing from $\tau$ to $\tau'$, then one should be able to carry this reasoning through limit steps and extend the whole procedure into the transfinite. Below we check how our approach fits into this picture.

In this section we work with theories extending $\ea$ (see \cite[Section 3.1]{leni22} for a definition). We fix an elementary presentation of a sufficiently large ordinal number $(\Gamma, \prec)$.\footnote{See for instance \cite[Chapter 4]{scwa11}.} In our progression two types of theories will be intertwined. Firstly we will have theories of the implicit commitment operator $\mathtt{IC}_{\lambda}$ formulated in the language of arithmetic $\lnat$ and two auxiliary primitive symbols: a unary predicate $\mathtt{I}(x)$ and binary relation $\mathcal{J}(x,y)$. The scheme of elementary induction is extended to all elementary formulae in the language with $\mathcal{J}$ and $\mathtt{I}$ (this theory is denoted with $\ea(\mathcal{J}, \mathtt{I})$). The intended interpretation of $\mathtt{I}$ is the set of sentences to which a mathematical agent is committed at the current stage of the process, given her initial acceptance of $\tau$. The intuitive reading of $\mc{J}(x,y)$ is ``$y$ is an implicit commitment of a theory $x$'' and we require that $x$ is presented via an elementary definition in the language $\lnat$ extended with the predicate $\mathtt{I}$. The second type of theories in the progression are theories in $\lnat$ denoted $\tau_{\alpha+1}$ (defined only for successor stages).

We extend our coding of syntax to cover the case of formulae with the newly added predicates. By default, metavariables $\sigma, \sigma', \tau,\tau'$ range over elementary presentations, in the extended language, of theories in $\lnat$. The set of such theories will be represented arithmetically via the predicate $\Delta_0(\mathtt{I})$, whereas the set of all formulae in the language with $\mathtt{I}$ via $\textnormal{Form}_{\mathtt{I}}$ To facilitate the reading, for a fixed $\sigma$, we often treat $\mathcal{J}(\sigma,y)$ as a set and write e.g. $\phi \in \mathcal{J}(\sigma)$ instead of $\mathcal{J}(\sigma, \phi)$. Finally, $\sigma\leq \sigma'$ denotes the arithmetical counterpart of the assertion ``All consequences of $\sigma$ are consequences of $\sigma'$''.

\begin{definition}

    \begin{align*}
            \mathtt{IC}_0:= &\;\ea(\mathtt{I}, \mathcal{J}) + \forall \phi\in\lnat\bigl(\mathtt{Prov}_{\tau}(\phi)\rightarrow \mathtt{I}(\phi)\bigr)&\\[1ex]
            \tau_{\alpha+1}:= &\;\{\phi\in\lnat : \mathtt{IC}_{\alpha}\vdash \phi\in \mathcal{J}(\mathtt{I})\}&\\[1ex]
            \mathtt{IC}_\lambda :=\;& \;\; \ea(\mathtt{I}, \mathcal{J}) + \forall \vphi \in \lnat\forall\beta\prec\lambda(\mathtt{Prov}_{\tau_{\beta+1}}(\vphi)\ra \mathtt{I}(\vphi)\bigr)\\
    \end{align*}
    Moreover each $\mathtt{IC}_{\alpha}$ is closed under the rules
    \begin{align}
    & \frac{\sigma \leq \sigma'}{\mc{J}(\sigma)\subseteq \mc{J}(\sigma')} \tag{\textsc{inv}}\\[1ex]
    &\frac{\forall x \,\sigma(\vphi(\dot{x}))}{\forall x \vphi\in\mc{J}(\sigma)} \tag{\textsc{ref}}
    \end{align}
\end{definition}

The intuition behind the definition above is that $\mathtt{IC}_{\alpha}$ describes the theory of the implicit commitment of $\tau$ at stage $\alpha$ (this theory is given by $\mathtt{I}$) and its implicit commitment (through $\mathcal{J}$). $\tau_{\alpha+1}$ describes the implicit commitment of level $\alpha+1$.

Now we give a more rigorous definition of this progression. We think of a formula $\mathtt{IC}(\alpha, x)$ as defining the theory describing the commitments of $\alpha$-th level. Formally this theory will be a set of pairs, whose first element encodes axioms and the second -- rules of reasoning in this theory. Let $\mathtt{IT}(\corn{\Phi(y,z)},\lambda, x)$ denote the formula 
\begin{multline*}
    (x)_0 = \corn{\forall \vphi\in\lnat\forall \beta\prec\lambda\bigl(\mathtt{Prov}_{\Phi(\beta,\hat{z})}(\vphi\in \mathcal{J}(\mathtt{I}))\rightarrow \mathtt{I}(\vphi)\bigr)} \wedge \\ \exists \sigma,\sigma'\in\Delta_0(\mathtt{I}) \bigl((x)_1 = \langle{ \corn{\sigma\leq \sigma'}, \corn{\mathcal{J}(\sigma)\subseteq \mathcal{J}(\sigma')}}\rangle \\ \vee \exists \vphi\in \mathtt{Form}_{\mathtt{I}}\exists\sigma\in\Delta_0(\mathtt{I})\bigl( (x)_1 = \langle{\forall x \sigma(\vphi(\dot{x})), \forall x\vphi\in\mathcal{J}(\sigma)}\rangle\bigr)\bigr) 
\end{multline*}

$\mathtt{IC}(\beta, x)$ is obtained via the standard diagonal lemma and provably in $\ea$ satisfies
\begin{equation*}
\mathtt{IC}(\beta, x)\equiv \mathtt{IT}(\corn{\mathtt{IC}(y,z)},\beta,x) 
\end{equation*}

Let us set $\tau_{\alpha}:= \{\phi\in\lnat \sth \mathtt{IC}_{\alpha}\vdash \phi\in \mathcal{J}(\mathtt{I})\}$ and $\tau_{\Gamma}:= \bigcup_{\alpha\prec\Gamma} \tau_{\alpha+1}$. We observe that $\mathtt{IC}_{\alpha}$ is given by an elementary formula. In order to facilitate the reading we shall use the same symbol $\vdash$ to denote both formalized and metatheoretical provability relation. In particular
\[\ea \vdash \pa\vdash \corn{\phi},\]
encodes $\ea\vdash \textnormal{Prov}_{\pa}(\corn{\phi}).$ This shall help us in the contexts where the provability predicates are nested. The expression
\[\tau\vdash \tau'\vdash \textnormal{RFN}(\sigma)\]
should be understood as
\[\tau \vdash \forall \phi \tau'\vdash\corn{\forall x \bigl(\textnormal{Prov}_{\sigma}(\phi(\dot{x}))\rightarrow \phi(x)\bigr)}.\]

We shall measure the strength of theories $\tau_{\alpha}$ by relating them to iterations of uniform reflection over $\tau$, defined in Section \ref{sec:prelim}. The following lemma shows that our main argument formalizes uniformly in $\ea$. It's proof is immediate.
\begin{lemma}\label{lem_unif_form_mainarg}
$\ea\vdash \forall \sigma\forall x \bigl((\mathtt{IC}_{\dot{x}}\vdash \corn{\sigma\subseteq \mathtt{I}}) \rightarrow (\tau_{x+1}\vdash \textnormal{RFN}(\sigma))\bigr)$
\end{lemma}
\begin{proposition}\label{prop_strength_iter}
    For every $\alpha<\omega^2$, $\tau_{\Gamma}\vdash \textnormal{RFN}^{\alpha}(\tau).$
\end{proposition}
\begin{proof}
    It is clearly sufficient to show that for every $\alpha<\omega^2$, $\ea$ proves that $\tau_{\alpha+1}\vdash \textnormal{RFN}^{\alpha+1}(\tau).$ By external induction on $n$, we prove that for each $n$ the following sentence is provable in $\ea$
    \begin{equation}\label{equat_ind_hyp}
        \tau_{\omega\cdot n+1}\vdash \textnormal{RFN}^{\omega\cdot n +1}(\tau).
    \end{equation}
    The step for $n=0$ follows from Lemma \ref{lem_unif_form_mainarg}, because $\mathtt{IC}_0\vdash \tau \subseteq \mathtt{I}$. 
    Assume inductively that $\ea$ proves that
    \[\tau_{\omega\cdot n +1}\vdash \textnormal{RFN}^{\omega\cdot n +1}(\tau).\] 
    By $\Sigma_1$-completeness and the fact that $\ea$-verifiably each $\mathtt{IC}_{\alpha}$ extends $\ea$ it follows that actually $\ea$ proves that
    \[\mathtt{IC}_{\omega\cdot n+1}\vdash\corn{\tau_{\omega\cdot n +1}\vdash \textnormal{RFN}^{\omega\cdot n +1}(\tau)}.\] 
    To verify the induction step we show $\forall x \theta(x)$, where
    \[\theta(x):= \mathtt{IC}_{\omega\cdot n+x+1}\vdash\corn{\tau_{\omega\cdot n+x +1}\vdash \textnormal{RFN}^{\omega\cdot n +x +1}(\tau)}.\]
    Firstly, for a sufficiently large natural number $k\in\mathbb{N}$ we prove $\forall l\leq k \theta(l)$ by external induction. This induction process does not immediately formalize in $\ea$, since ``$\mathtt{IC}_{\alpha}\vdash \cdots$'' is a $\Sigma_1$-formula. To fix this we need to control the sizes of proofs in $\mathtt{IC}_{\alpha}$. Let $\sigma\vdash^{y}\phi$ denote the arithmetical formula which says that there is a proof of $\phi$ in the theory $\sigma$ whose total size is less than $y$, where by a total size of a proof $p$ we understand the length of a binary sequence coding $p$. Secondly, for all $x\geq k$ we prove $\theta(x)$ by elementary induction inside $\ea$ for the formula
    \[\theta^f(x):=\mathtt{IC}_{\omega\cdot n+x+1}\vdash^{f(x)}\corn{\tau_{\omega\cdot n+x +1}\vdash \textnormal{RFN}^{\omega\cdot n +x +1}(\tau)},\]
    to show that $\ea\vdash \forall x\geq k \theta(x),$ where $f(x) = x^{x^x}$.  From now on we reason in $\ea$. $\theta(k)$ holds by our previous induction assumption. Fix any $x$ and assume $\theta^f(x)$ holds. In particular, by provable $\Sigma_1$-completeness there is a function $g_1$ such that
    \[\mathtt{IC}_{\omega\cdot n +x+2}\vdash^{g_1(f(x))}\corn{\theta(\num{x})}.\]
    By applying the axiom of $\mathtt{IC}_{\omega\cdot n + x+1}$ we see that
    \[\mathtt{IC}_{\omega\cdot n +x+2}\vdash^{g_2(g_1(f(x)))}\corn{\mathtt{IC}_{\omega\cdot n +x+1}\vdash \corn{\textnormal{RFN}^{\omega\cdot n + x +1}(\tau)\subseteq \mathtt{I}}}.\]
    By using Lemma \ref{lem_unif_form_mainarg} inside $\mathtt{IC}_{\omega\cdot n +x+2}$ we see that
    \[\mathtt{IC}_{\omega\cdot n +x+2}\vdash^{g_3(g_2(g_1(f(x))))}\corn{\tau_{\omega\cdot n +x+2}\vdash \textnormal{RFN}^{\omega\cdot n + x +2}(\tau).}\]
    To complete the induction step, it is enough to show inside $\ea$ that
    \[g_3(g_2(g_1(f(x))))\leq f(x+1).\]
    However, $g_1(y)$ can be taken to be $y^c$ for some constant $c$, since the verification that an object of size  $y$ is a proof in a finitely axiomatized theory is in P-time and we can assume that $k$ is large enough. $g_2(y)$ can be taken to be $y+cy$ since the corresponding proof consists in extending the given proof with an independent of $y$ number of formulae of length at most $y$. The same is true of $g_3(y)$. Hence in total we have that
    \[g_3(g_2(g_1(f(x))))\leq y^{d},\]
    for some constant $d$ independent of $y$, and for all large enough numbers $y$. It is clear that if we choose $f(x)$ to be $x^{x^x}$, then for all large enough $x$ we have $\left(x^{x^x}\right)^d\leq x^{d\cdot x^x}\leq (x+1)^{{(x+1)}^{x+1}}.$ 
    
    Consequently, we have that $\ea\vdash \forall x \tau_{\omega\cdot n +x+1} \vdash \textnormal{RFN}^{\omega\cdot n +x+1}(\tau)$. In particular $\mathtt{IC}_{\omega\cdot (n+1)}\vdash \forall x \tau_{\omega\cdot n +x+1} \vdash \textnormal{RFN}^{\omega\cdot n +x+1}(\tau)$. Hence \[\mathtt{IC}_{\omega\cdot(n+1)}\vdash\forall x\forall \phi \mathtt{I}(\corn{\forall y(\textnormal{Prov}_{\textnormal{RFN}^{\omega\cdot n + \dot{x}}}(\phi(\dot{y}))\rightarrow \phi(y))}).\]
    By applying the (REF) rule, we obtain that 
    \[\mathtt{IC}_{\omega\cdot(n+1)}\vdash\forall \phi \corn{\forall x\forall y(\textnormal{Prov}_{\textnormal{RFN}^{\omega\cdot n + \dot{x}}}(\phi(\dot{y}))\rightarrow \phi(y))}\in  \mathcal{J}(\mathtt{I}).\]
    By $\Sigma_1$-completeness, the above is provable in $\ea$. As a consequence, $\ea$ proves
    \[\forall \phi \tau_{\omega\cdot(n+1)+1}\vdash \forall x\forall y(\textnormal{Prov}_{\textnormal{RFN}^{\omega\cdot n + \dot{x}}}(\phi(\dot{y}))\rightarrow \phi(y)).\]
     It follows that $\ea\vdash \tau_{\omega\cdot(n+1)+1}\vdash \textnormal{RFN}^{\omega\cdot(n+1)+1}(\tau),$ which concludes the induction step of the main external induction and the whole argument.
\end{proof}

Arguably, one would expect the implicit commitments of a theory to contain more iterations of uniform reflection than just $\omega^2$. We conjecture that the proof of Proposition \ref{prop_strength_iter} can be formalized in $\ea$: we expect that what can be verified internally in $\ea$ is in fact that $\tau_{\omega^n+1} \vdash \textnormal{RFN}^{\omega^n+1}(\tau)$. In this way we could conclude that $\tau_{\Gamma}$ yields at least all iterations of uniform reflection up to the level $\omega^{\omega}$ (and perhaps more). We leave the verification of this to further research.

\section{Implicit Commitment and Domain Constants}\label{sec:domcon}

The theory presented in \cite{leni22}, as well as its iteration into the transfinite just described, are formulated in the language of arithmetic $\lnat$. Given the foundational relevance of the notion of implicit commitment, it is natural to ask whether it is possible to apply directly the framework above to foundationally relevant theories formulated in different languages. 

\subsection{Some positive results} An obvious case study is set theory. Standard systems of set theory, such as $\mathtt{ZFC}$, are formulated in a term-free language $\lin$ with signature $\{\in\}$. However, to apply our framework for implicit commitment -- \textsc{(ref)}, in particular -- to $\mathtt{ZFC}$ so formulated  one would require names for objects of the domain; in fact,  having names for the finite ordinals would already deliver nontrivial implicit commitments. For example, let $\mc{L}^{\vtl}_\in$ be the expansion of $\lin$ by a constant $0$ and a binary function symbol $\vtl$, whose informal interpretation is $x\cup \{y\}$. Let $\mathtt{ZFC}^{<\omega}$ feature the standard axioms of $\mathtt{ZFC}$, except that the empty set axiom is replaced by 
\[
    z=0\lra \forall y (y\notin z),
\]
and that we have additional axiom
\[
    z= y\vtl y \lra \forall u (u\in z \lra u\in y \vee u=y). 
\]
\noindent Working in a weak fragment $M$ of $\mathtt{ZFC}^{<\omega}$ -- in fact, we can conveniently choose a finite set theory equivalent to $\ea$ \cite{pet09} so that we can safely transfer here some of our results obtained in the arithmetical setting -- we can define by recursion on $\omega$ a function
\begin{align*}
    &\name(0)=0\\
    &\name(n\vtl n)= \lan \corn{\vtl}, \name(n),\name(n)\ran,
\end{align*}
where $\corn{\vtl}$ is a code for the symbol $\vtl$. This metatheory $M$ can thus establish that
    \begin{align*}
         & \forall x \in \omega \exists ! y \,\name(x)=y,\\
        & \forall x\forall y(\name(x)=\name(y)\ra x=y).
    \end{align*}
As before, we abbreviate the naming function via the dot notation. All syntactic notions and operations should now be understood via $M$. For readability, we still use the same labels for the usual syntactic predicates; we also abbreviate $n,m,...$ as variables for finite ordinals in the sense of $M$. Now, since $M$ gives us that $\mathtt{Prov}_{\mathtt{ZFC^{<\omega}}}(\corn{\mathtt{Con}(\dot{n},\mathtt{ZFC}^{<\omega})})$, we have that $\mathtt{ZFC}^{<\omega}$ as described above and its re-axiomatization via restricted consistency
\[
    x\in \mathtt{ZFC_I^{<\omega}} :\lra x\in \mathtt{ZFC}^{<\omega}\vee \exists n\, x=\corn{\mathtt{Con}(\dot{n},\mathtt{ZFC}^{<\omega})}
\]  
are ``elementarily'' reducible to one another.\footnote{Again, we are employing set-theoretic functions corresponding to the elementary ones. More details shortly.}  We are here assuming that our axiomatizations can be expressed as bounded formulae in $M$, and that some properties of $\ea$ carry over to $M$: in particular, we are resorting to a  witnessing theorem for  $\Pi_2$-statements -- cf.~\cite[Thm.~5.2]{pet09} -- in the sense of $\mc{L}^{<\omega}_\in$, which provides the required link between the notion of elementary reducibility employed in \cite{leni22} and $M$-reducibility.  If one reformulates \textsc{(ref)} as 
\begin{align}
    \tag{\textsc{ref*}} &M\vdash \forall y(y \text{ is a finite ordinal } \ra \mathtt{ZFC}^{<\omega}(\corn{A(\dot{y})})\\
    \notag &\Ra \forall y(y \text{ is a finite ordinal } \ra A) \in \mc{I}(\mathtt{ZFC}^{<\omega}),
\end{align}
and relativizes \textsc{(inv)} to the proof-transformations available in $M$ -- let's call it \textsc{(inv$^M$)} -- the same argument employed in Proposition \ref{prop:basicleni} gives 
\begin{proposition}
    $\mathtt{Con}(\mathtt{ZFC}^{<\omega})$ is in the implicit commitments of  $\mathtt{ZFC}^{<\omega}$ defined via \textsc{(inv$^M$)} and \textsc{(ref*)}.
\end{proposition}
Direct quantification over natural number terms enable us to obtain even stronger implicit commitments. Given the equivalence of  $\Pi_1$-Uniform and $\Pi_1$-Local Reflection,\footnote{We stated this above for $\forall\Delta_1^b$-reflection, but the proof applies with no modification to more standard classes of formulae.} and that the hierarchy of local reflection over theories extending our metatheory $M$ is strictly increasing in terms of logical strength -- $\mathtt{Rfn}_{\Pi_{n+1}}(\tau)$ strictly extends $\mathtt{Rfn}_{\Pi_{n}}(\tau)$, for $n\geq 1$ \cite{bek05} -- full local reflection is substantially stronger than mere consistency. Consider now the re-axiomatization of $\mathtt{ZFC}^{<\omega}$ via local reflection
\[
    x\in \mathtt{ZFC_{II}^{<\omega}} :\lra x\in \mathtt{ZFC}^{<\omega}\vee \exists n,\vphi\; x=\corn{\mathtt{Proof}_{\mathtt{ZFC}^{<\omega}}(\dot{n},\corn{\vphi})\ra\vphi}
\] 
In this formula, both $n$ and $\vphi$ range over finite ordinals in the sense of the metatheory $M$. 
Again, the argument outlined in Proposition \ref{prop:basicleni} can be adapted to obtain
\begin{proposition}
    $\mathtt{Rfn}({\mathtt{ZFC}^{<\omega}})$ is in the implicit commitments of $\mathtt{ZFC}^{<\omega}$ defined via \textsc{(inv$^M$)} and \textsc{(ref*)}. 
\end{proposition}

 The previous observations cannot immediately be extended to full Uniform Reflection, since the relevant application of Proposition \ref{prop:basicleni}, and of \textsc{(ref)} in particular, in the proof of Uniform Reflection requires quantification over names for the entire domain of quantification. However, we can achieve restricted versions that still extend Local Reflection. Let 
\[
    \tag{$\mathtt{RFN}_\nat(\tau)$} \forall n \big(\mathtt{Prov}_\tau(\corn{\vphi(\dot{n})})\ra \vphi(n)\big)
\]
be a restricted version of Uniform Reflection that only quantifies over  finite ordinals. By the usual analysis,  $\mathtt{RFN}_\nat(\tau)$ is still stronger than $\mathtt{Rfn}(\tau)$, for $\tau$ that -- just like $\mathtt{ZFC}^{<\omega}$ -- {does not prove any false  $\Sigma_1$-claims in the language restricted to the finite ordinals}. 
\begin{proposition}
\textsc{(inv$^M$)} and \textsc{(ref*)} imply that $\mathtt{RFN}_\nat(\mathtt{ZFC}^{<\omega})$ is in the implicit commitments of $\mathtt{ZFC}^{<\omega}$.
\end{proposition}

\subsection{Clouds on the horizon}
As mentioned, a general formulation of Uniform Reflection would require names for all objects in $\mbb{V}$ as formal objects in the scope of our of quantifier. Fortunately, analogous demands are customary in certain approaches to the metamathematics of set theory.  Let $\linfty$ be defined as in \cite{dev17,fuj12} and $T\supseteq \kpomega$, where $\kpomega$ is Kripke-Platek set theory with the axiom of infinity as defined for instance in \cite{bar17}. {We need such a $T$ for the absoluteness of syntactic notions (in fact, only for the satisfaction relation, otherwise we need much less).} The usual syntactic notions become $\Delta_1^{{\kpomega}}$, hence absolute with respect to all transitive models. Again, we employ the same expressions to denote syntactic notions relative to the new metatheory.  The external language is just $\lin$, whereas the formalized language is $\linfty$, featuring a constant $\dot{x}$ for each external object $x$. Uniform Reflection then becomes:
\begin{equation}
    \tag{$\mathtt{RFN}^\infty(T)$} \{\forall x (\mathtt{Prov}_T(\corn{\vphi(\dot{x})})\ra \vphi(x))\sth \vphi(x)\in \lin\}
\end{equation}
It's important to notice that, even though $\mathtt{ZFC}$ is seen by our metatheory $\mathtt{KP}_\omega$ to be formulated in $\linfty$, its axioms do not directly employ the new constants. This enables one to have 
\begin{lemma}\label{lem:gencon}
For $T\supseteq\kpomega$ a $\linfty$-theory, we have that  $\kpomega\vdash \forall x\,\mathtt{Prov}_{T}(\corn{\vphi(\dot{x})})\ra \mathtt{Prov}(\corn{\forall x \,\vphi})$.  
\end{lemma}
\begin{proof}
Constants do not appear in non-logical axioms of $T$. It is an admissible rule of first-order logic that $\Gamma \vdash \vphi(c/v)$ only if $\Gamma \vdash \forall x\vphi$ for $c$ not occurring in $\Gamma$. The claim is the formalization in $\kpomega$ of this rule.
\end{proof}
As an immediate corollary, we obtain:
\begin{corollary}
    For $T\supseteq \kpomega$ as above, the distinction between Uniform and Local Reflection collapses. 
\end{corollary}

The previous results cast doubt on the prospect of applying our framework for implicit commitment to study the implicit commitment of theories formulated in languages augmented with ``domain constants''. For definiteness, we keep considering theories in $\linfty$ extending $\kpomega$, but our considerations are likely to generalize to analogous settings. 

It is unproblematic to formulate a suitable notion of proof-theoretic reducibility. Instead of p-time or elementary reducibility, we could consider proof-transformations in $\kpomega$ itself. Let $\tau,\tau'$ be $\Delta_1^{{\kpomega}}$-presentations of theories. We let:
\begin{align*}
    &\tau\leq_{\kpomega}\!\tau' :\Lra   \kpomega\vdash \mathtt{Proof}_\tau (y,x)\ra \mathtt{Proof}_\tau'(f(y),x)
\end{align*}
where $f$ is $\kpomega$-definable. Similarly, the idea that implicit commitments are preserved by one's preferred notion of proof-theoretical reducibility can be adequately formulated in the new framework as
\[
 \tag{\textsc{inv$^\infty$}} \tau \leq_{\kpomega}\! \tau' \Ra \mc{I}(\tau')\subseteq \mc{I}(\tau)
\]
Unsurprisingly, troubles arise when we consider the impact of adding the reflection principle
\[
  \tag{\textsc{ref$^\infty$}} \kpomega \vdash \forall x\, \tau (\corn{\vphi(\dot{x})}) \Ra \forall x\vphi \in \mc{I}(\tau).
\]
In light of Lemma \ref{lem:gencon}, we cannot hope to get a substantial logical strength, as \textsc{(ref$^\infty$)} becomes admissible under minimal assumptions:
\begin{proposition}\label{prop:admrul}
    Let $\tau$ be $\Delta_1^{\kpomega}$. Suppose that if 
    $
        \kpomega\vdash \mathtt{Prov}_\tau(\corn{\vphi})
   $
    also $\vphi\in \mc{I}(\tau)$. Then \textsc{(ref$^\infty$)} becomes an admissible rule. 
    
\end{proposition}
\begin{proof}
    Assume that  $\kpomega \vdash \forall x\, \tau (\corn{\vphi(\dot{x})})$. Then also  $\kpomega \vdash \forall x\, \mathtt{Prov}_{\tau} (\corn{\vphi(\dot{x})})$. By Lemma \ref{lem:gencon}, $\kpomega\vdash \mathtt{Prov}_\tau(\corn{\forall x \vphi})$. By our assumption on the closure of $\mc{I}$ under provability, we get that $\forall x \vphi \in \mc{I}(\tau)$, as wanted.
\end{proof}

We dubbed the assumptions in Proposition \ref{prop:admrul} `minimal'. This claim can be made precise by noticing that a set $\mc{I}$ satisfying the assumption in the Proposition is the set of $\tau$-provable formulae. Therefore, since this interpretation is also sufficient to validate \textsc{inv$^\infty$}, we have
\begin{corollary}
    The commitments $\mc{I}(\tau)$ defined only by the assumptions in Proposition 
    \ref{prop:admrul}, \textsc{(inv$^\infty$)} and \textsc{(ref$^\infty$)} are reducible to $\tau$ itself. 
\end{corollary}
\begin{remark}
When interpreting $\mc{I}(\tau)$ in $\tau$, we are using `reducible' instead of `relatively interpretable', because strictly speaking $\mc{I}(\tau)$ is not part of the signature. 
\end{remark}

So far we considered theories formulated in $\linfty$. It may be objected that this is not general enough to discourage the development of a theory of implicit commitments via domain constants expansions.  However, it's important to emphasize that the  results of this section generalize to any theory formulated in a language featuring expansions with domain constants: it's only required that the new constants do not appear in the axioms of the theories whose implicit commitment is under scrutiny. 

We consider these results as convincing arguments against developing an adequate theory of implicit commitment in the setting with domain constants, and move to alternative proposals. 

%





\section{Generalizing Invariance and Reflection}\label{sec:genict}

In this section we introduce a more promising generalization of the theory of implicit commitment from \cite{leni22}. As it will be clear  shortly, much of the discussion will rotate around choosing the right notion of invariance. In effect,  much of the discussion  will involve a notion of implicit commitment for arbitrary first-order theories $\tau$, \emph{given a fixed intepretation of $\sot$ in it}. Equivalently, one can think of implicit commitments as relative to pairs $(\tau,N)$, where $N\colon \sot \to \tau$. 

\subsection{Invariance}

When trying to generalize the notion of implicit commitment, one faces forces that pull in opposite directions. On the one hand, it is reasonable to require that one can compare implicit commitments of theories that are formulated in different languages -- but that are nonetheless logically comparable. On the other hand, it is also reasonable not to impose that the implicit commitments of one theory encompass statements that belong to an open-ended class of languages. 
To liberalize \textsc{(inv)} to $\tau$ and $\tau'$ formulated in different signatures, there are two parameters to consider. The first is the notion of proof-theoretic reducibility appearing in the antecedent of \textsc{(inv)}, the second is the relation of inclusion employed in its consequent. As to the former, one can resort to generalizations of proof-theoretic reducibility. Relative interpretability naturally suggests itself. As a first approximation, one might consider the principle
\[
    \tag{\textsc{inv0}} \tau \vtl \tau' \Ra \mc{I}(\tau)\subseteq \mc{I}(\tau'), 
\]
where $\tau \vtl \tau'$ abbreviates `there is a $K\colon \tau \to \tau'$'.
It's clear that \textsc{inv0} is a non-starter. Consider the following classical result:
\begin{lemma}[Feferman]\label{lem:fefinc}
    Let $\tau$ be a $\Delta_0^b$-theory. Suppose there is a relative interpretation $N\colon \sot \to \tau$. Then there is a $K\colon \tau +\neg \mathtt{Con}^N(\tau)\to \tau$. 
\end{lemma}
\noindent \textsc{inv0}, if taken seriously, would entail that $\tau$ would display a canonical statement of its inconsistency among its implicit commitments (assuming, of course, that $\mc{I}(\tau)$ contains $\tau$). 

One can try to replace the relation of inclusion in the consequent of \textsc{inv0}  with relative interpretability:
\[
    \tag{\textsc{inv1}} \tau \vtl \tau' \Ra \mc{I}(\tau)\vtl \mc{I}(\tau').
\]
This version of invariance would be inadequate because it clashes with our original motivation of generalizing the framework introduced in Section \ref{sec:basic}.  
 In particular, we would like \textsc{inv} to be a special case of \textsc{inv1} so that, when coupled with a suitable notion of \textsc{reflection}, Proposition \ref{prop:basicleni} could be obtained for the degenerate case of theories in the same signature and ordered by simple inclusion. But now consider $\pa+\neg\mathtt{Con}(\pa)$. By Proposition \ref{prop:basicleni}, 
 \beq\label{eq:inconnconpa}
    \mathtt{RFN}(\pa+\neg \mathtt{Con}(\pa))\in \mc{I}(\pa+\neg\mathtt{Con}(\pa)).
 \eeq
 Therefore, $\mc{I}(\pa+\neg\mathtt{Con}(\pa))$ is inconsistent. By applying Lemma \ref{lem:fefinc} and $\textsc{inv1}$ we would then obtain that $\mc{I}(\pa)$ is inconsistent because it interprets an inconsistent theory.  

Even if one moved to the stricter notion of \emph{faithful interpretability} (cf.~\S\ref{sec:thre}), we would not be able to overcome this potential unsoundness of implicit commitments. Consider the principle
\[
    \tag{\textsc{inv2}} \tau \vtl_f \tau' \Ra \mc{I}(\tau)\subseteq \mc{I}(\tau'),
\]
where $\tau\vtl_f \tau'$ now abbreviates `there is a faithful $K\colon \tau \to \tau'$'.  It is known that there are widely employed theories, such as $\pa$ that faithfully interpret their own inconsistency -- see \cite[Thm.~5.5.]{vis05}. So, once again, \textsc{(inv2)} would entail the unsoundness of the implicit commitments of theories such as $\pa$. Moreover, the reasoning employed to rule out \textsc{(inv1)} can also be employed against 
    \[
    \tag{\textsc{inv3}} \tau \vtl_f \tau' \Ra \mc{I}(\tau)\vtl_f \mc{I}(\tau').
\]

Since we aim to generalize \emph{necessary} conditions for implicit commitment, it's reasonable to move to strict notions of theoretical reducibility that do not allow for the reduction of unsoundness assertions.  The main intuition that we follow is that theories that are ``the same'' -- in an adequate formal sense -- should have ``the same'' commitments. Since we consider bi-interpretability (cf.~Introduction) between theories as a robust notion of sameness of theories, we employ it in our preferred generalization of invariance. Some work is required, however, to realize this intuition.  An immediate issue is that \textsc{inv} concerns ``sameness'' of commitments only in a derivative sense, as the consequence of the combination of two inclusions of commitments (and of two relations of proof-theoretic reducibility). The immediate analogue of the notion of proof-theoretic reduction when theory-identity is intended via bi-interpretability is the notion of \emph{retract} introduced in Section \ref{sec:thre}.

When attempting to reformulate invariance criteria by means of the notion of retract, one can immediately see  that the principle
\[
\tag{\textsc{inv4}} \tau \vtl_r \tau' \Ra \mc{I}(\tau) \subseteq \mc{I}(\tau'),
\]
where $\tau \vtl_r \tau'$ denotes the relation occurring when $\tau'$ is a retract of $\tau$, is not adequate. Although the unsoundness charges to the previous, attempted generalizations of invariance are addressed, \textsc{(inv4)} suffers from other shortcomings.
\begin{obse}
Under \textsc{(inv4)}, there are consistent $\tau,\tau'$ that are bi-interpretable but whose (identical) implicit commitments are inconsistent. 
\end{obse}
\begin{proof}
A simple example involves theories in the signature $\{P\}$, for $P$ a unary predicate letter, and such that the only axiom of $\tau$ is $\forall x P x$, and the only axiom of $\tau'$ is $\forall x \neg P x$. The interpretation witnessing the bi-interpretability (in both directions) simply sends $P$ to $\neg P$. 

Under the assumption that, for an arbitrary $\sigma$, $\mc{I}(\sigma)$ includes the nonlogical axioms of $\sigma$ at least, \textsc{(inv4)} entails that $\mc{I}(\tau)=\mc{I}(\tau')$ and  therefore that $\mc{I}(\tau)$ is inconsistent. 
\end{proof}
\begin{remark}
\cite{nic22} discusses more principled cases of a similar kind. Two common variants of Kripke-Feferman truth, one whose truth predicate is consistent, the other is complete. The theories are mutually inconsistent, due to the Liar paradox. They are also bi-interpretable (and, in fact, definitionally equivalent). 
\end{remark}

An obvious variant of \textsc{(inv4)} is inspired by the principle that bi-in\-ter\-pre\-ta\-ble theories should have \emph{bi-interpretable} commitments. When breaking down this intuition,  it would result in the principle that $\tau$'s being a retract of $\tau'$ is sufficient for the commitments of $\tau'$ to be retractions of the ones of $\tau$. 
\[
    \tag{\textsc{inv$^+$}} \tau' \vtl_r \tau \Ra \mc{I}(\tau')\vtl_r \mc{I}(\tau). 
\]
Unfortunately, this is not yet a direct extension of the original \textsc{(inv)}. The original \textsc{(inv)} displayed a uniform reduction on both sides of the conditional; through the lens of relative interpretability, the very same (identity) interpretation holds on both sides. By contrast, \textsc{(inv$^+$)} allows for different interpretations to be employed. 

\begin{example}\label{exa:ctaca}
The theory $\mathtt{CT}$ (cf.~Section \ref{sec:atot}) is a retract of $\mathtt{ACA}$, where {$\mathtt{ACA}$ is the extension of the well-known predicative system $\mathtt{ACA}_0$ with full induction for arbitrary formulae of the language $\mc{L}_2$ of second-order arithmetic.}\footnote{For a published proof of this result, due Ali Enayat and Albert Visser, see \cite{nic17}.} The interpretations that are employed in the result are: 
\begin{itemize}
    \item $\mathtt{K}\colon \mathtt{CT} \to \mathtt{ACA}$, which leaves the arithmetical vocabulary unchanged, and such that $\T^K x$ is `there is a truth class $X$ for $x$ and $x\in X$'. 
    \item $\mathtt{L} \colon \mathtt{ACA}\to \mathtt{CT}$, which also preserves the arithmetical primitives, and such that $x\in^\mathtt{L} X$ is `$X$ is a formula of $\lnat$ with one free variable and $X$ is true of $x$'. 
\end{itemize}   
In studying the implicit commitment of $\mathtt{CT}$ and $\mathtt{ACA}$, we would like to keep the interpretations fixed, not least because we would like to preserve their nice feature of keeping potential implicit commitments in $\lnat$ fixed. 
\end{example}

The example motivates a notion of invariance of implicit commitments that  preserves the relevant interpretations. We write $\tau \vtl_r^{\mathtt{K},\mathtt{L}} \tau'$ for `the interpretations $\mathtt{K}$ and $\mathtt{L}$ witness that $\tau'$ is a retract of $\tau$'. 
\[
    \tag{\textsc{uniform invariance}} \tau \vtl_r^{\mathtt{K},\mathtt{L}} \tau' \Ra \mc{I}(\tau)  \vtl_r^{\mathtt{K},\mathtt{L}} \mc{I}(\tau').
\]
Obviously, \textsc{uniform invariance} -- henceforth, \textsc{(ui)} -- is not sufficient for obtaining non-trivial implicit commitments.  Again, by letting $\mc{I}(\tau)$ to be $\tau$ itself (for arbitrary $\tau$), we obtain
\begin{lemma}
    \textsc{Uniform invariance} can be conservatively interpreted in any $\tau \supseteq \sot$. 
\end{lemma} 
\textsc{Uniform invariance} is, however, still not enough. The problem of inconsistent implicit commitments resurfaces again.\footnote{The following lemma was already observed in \cite{vis06} as a corollary of a more general phenomenon.}
\begin{lemma}\label{lem:retncon}
$\pa+\neg\mathtt{Con}(\pa)$  is a retract of $\pa$.
\end{lemma}
\begin{proof}
    Let $\mathtt{H}_T\colon T\to \mathtt{PA}+\mathtt{Con}(T)$ be the Henkin-Feferman interpretation  given by the arithmetized completeness theorem (see e.g.~\cite[p.~77]{lin17}). Since, by G\"odel's second incompleteness theorem $\mathtt{Con}(\pa+\neg \mathtt{Con}(\pa))$ is derivable in $\pa+\mathtt{Con}(\pa)$,  one can define an interpretation
    \[
        \mathtt{I}(A):= (\mathtt{Con}(\pa)\land \mathtt{H}_{\pa+\neg\mathtt{Con}(\pa)}(A))\vee (\neg\mathtt{Con}(\pa)\land \mathtt{id}_\pa(A)).
    \]
    Reasoning in $\pa+\neg\mathtt{Con}(\pa)$, one can show that the trivial isomorphism given by identity gives, for any primitive $P(x)$ of $\lnat$, that 
    \[
          P(\vec{x})\lra P^{\mathtt{id}\circ \mathtt{I}}(\vec{x}).
    \]
\end{proof}
From Lemma \ref{lem:retncon} and \textsc{(uniform invariance)}, we require that
\beq
    \mc{I}(\pa)\vtl^{\mathtt{id},\mathtt{I}}_r \mc{I}(\pa+\neg \mathtt{Con}(\pa)). 
\eeq
However, we have seen that  $\mc{I}(\pa+\neg\mathtt{Con}(\pa))$ is bound to be inconsistent by \eqref{eq:inconnconpa}. Since retracts require mutual interpretability, again we are obtaining that $\mc{I}(\pa)$ is inconsistent.

It's clear what is going on all along: yes, we aim for a strict notion of ``sameness'' of commitments; yes, we require that the \emph{kinds} of interpretations of basic concepts are preserved from theories to commitments; but we also require that the interpretations involved \emph{agree on their basic domains of numbers/syntactic objects}. This is what fails to happen in  the case just considered. The interpretation $\mathtt{H}$, built in the interpretation $\mathtt{I}$ does not preserve the meaning of the natural numbers wrt to $\mathtt{id}$. 

We are led to a notion -- on which we finally settle -- of invariance that refines \textsc{uniform invariance} by requiring that the ontology of the natural numbers in the sense of our basic syntactic theory is preserved in the interpretations. For instance, in Example \ref{exa:ctaca}, the two theories include $\pa$, and the interpretation of $\lnat$ is just the identity interpretation for both $\mathtt{K}$ and $\mathtt{L}$. We generalize this scenario. 
\begin{definition}\label{dfn:adeint}
The interpretations $N\colon \sot \to \tau$,  $M\colon \sot \to \tau'$, $F \colon \tau \to \tau'$ and $G\colon \tau'\to \tau$ are \emph{adequate} if the following diagrams commute:
\medskip

\begin{minipage}[c]{.45\textwidth}
\[
\begin{tikzcd}
\tau \arrow[r, "F"] & \tau'\\
\sot \arrow[u, "N"] \arrow[ru, "M"]
\end{tikzcd}
\]
\end{minipage}
\begin{minipage}{.45\textwidth}
\[
\begin{tikzcd}
\tau' \arrow[r, "G"] & \tau\\
\sot \arrow[u, "M"] \arrow[ru, "N"]
\end{tikzcd}
\]
\end{minipage}

\noindent In other words, we are requiring that the interpretation $F\circ N$ is $\tau'$-provably \emph{identical} -- in the sense of \S\ref{sec:thre} -- to $M$, and symmetrically for, $\tau$, $G \circ M$, and $N$.\footnote{We conjecture that equality of interpretations can be liberalized to isomorphism of interpretations, in such a way that the results below will carry over to this more liberal setting. }

We write $\tau \btrl^{F,G}_r \tau'$ for: ` verifiably in $\sot$, $\tau'$ is a retract of $\tau$ witnessed by the adequate interpretations $F \colon \tau \to \tau'$ and $G\colon \tau'\to \tau$'. 
\end{definition}

Finally, we settle for the following notion of invariance, taking into accounts all the adjustments to the starting, naive intuition about sameness of commitments. 
\[
  \tag{\textsc{uniform $\nat$-invariance}} \tau \btrl^{F,G}_r \tau' \Ra \mc{I}(\tau)\vtl^{F,G}_r \mc{I}(\tau')
\]
As before, it's clear that \textsc{uniform $\nat$-invariance} does not rule out trivial implicit commitments $\tau=\mc{I}(\tau)$. 
\begin{lemma}
    \textsc{Uniform $\nat$-invariance} can be conservatively interpreted in any $\tau$ interpreting  $\sot$. 
\end{lemma}
We now turn to the other component of  our framework for implicit commitment. 

\begin{remark}\label{rem:smooth1}
    We have chosen an asymmetric version of \textsc{uniform $\nat$-invariance}, whose left-hand side requires a $\sot$-provable retract relation, whereas its right hand side doesn't. This reflects the asymmetry already present in the original principle \textsc{(inv)}. We could require $\sot$-verifiability on both sides, and the results below will carry through. We will comment on this in due course. 
\end{remark}

\subsection{Reflection}

The principle \textsc{(ref)} studied in \cite{leni22} states that if our metatheory establishes that all numeral instances of $\vphi$ are in the axiom ``set'' $\tau$, then $\forall x \vphi$ is in the implicit commitments of $\tau$. As already emphasized, it's then clear that the meaningfulness of the principle rests on the numerals being part of the overall domain of quantifiers of the metatheory, of $\tau$, and of its implicit commitments. To generalize \textsc{(ref)}, we then need to make sure to preserve the idea that numerals ``make sense'' in the meta-theory, even though $\tau$ may be formulated in a language that does directly feature resources to name natural numbers. The previous discussion suggests an obvious solution and constraints substantially the space of possible proposals; in fact, we could only come up with one natural proposal, which we now describe. 

Since we are studying implicit commitments of theories $\tau$ for which there is an interpretation $N\colon \sot \to \tau$, it's natural to recover the required domain of numbers needed to formulate our generalization of \textsc{(ref)} via $N$.  The generalized principle of reflection is then:
\[
\tag{\textsc{generalized reflection}}\label{eq:ref2}
    \tau \vdash \forall x \coltyp N\, \tau^N(\corn{\vphi(\dot{x})})\Ra \forall x\coltyp N \,\vphi(x)\in \mc{I}(\tau)
\]
 It should be clear that, in \ref{eq:ref2} -- henceforth, \textsc{(gr)} -- the meaning of the G\"odel quotes -- and, more generally, of the arithmetized syntactic apparatus -- is provided via the interpretation $N$. 

Just like \textsc{(invariance)}, \textsc{(ref)}, and \textsc{(ui)}, the principle \textsc{(gr)} does not force any logical strength.
 \begin{lemma}
    For $\tau$ schematic and interpreting $\sot$, \textsc{(gr)} can be conservatively interpreted. 
 \end{lemma}
 \begin{proof}
Let's fix an interpretation $N\colon \sot \to \tau$, and define
\[
    \mc{I}(\tau)= \{ \forall x \coltyp N \,\vphi \sth \tau \vdash \forall x \coltyp N \,\tau(\corn{\vphi(\dot{x})})\}.
\]
Let $\mathtt{SC}[\tau]$ be as above.
%
%
 If $\tau \vdash (\forall x:N)\,\tau(\corn{\vphi(\dot{x})})$ for some $\vphi$, then by $\tau \subseteq \T$, also $\mathtt{SC}[\tau]\vdash (\forall x:N)\,\T(\corn{\vphi(\dot{x})})$ and therefore $\mathtt{SC}[\tau]\vdash (\forall x:N)\,\vphi(x)$. So, $\mc{I}(\tau)\subseteq \mathtt{SC}[\tau]$. 

The proof is completed by Lemma \ref{lem:sccons}.

 \end{proof}

 We can finally define our generalized necessary conditions for implicit commitment. 
 \begin{definition}
    For $\sigma$ such that $N\colon \sot \to \sigma$, the collection $\mc{I}(\sigma)$ of its implicit commitments is constrained by the following: for $\tau,\tau'$  theories interpreting $\sot$:
    \begin{align}
     \tag{\textsc{uniform $\nat$-invariance}}& \tau \btrl^{F,G}_r \tau' \Ra \mc{I}(\tau)\vtl^{F,G}_r \mc{I}(\tau')\\
    \tag{\textsc{generalized reflection}}&
    \tau \vdash \forall x \coltyp N\, \tau^N(\corn{\vphi(\dot{x})})\Ra \forall x\coltyp N \,\vphi(x)\in \mc{I}(\tau)
    \end{align}
 \end{definition}

\noindent  In the following section, we study the main properties of generalized commitments.

 \subsection{Uniform $N$-reflection}\label{sec:uniref}

In Section \ref{sec:basic} we employed the Uniform Reflection principle for $\tau$ -- with $\tau$ a theory in the language of arithmetic -- both as a test of the non-trivial strength of our principles for implicit commitments, and as natural lower bound for them.  To test the adequacy of the generalized framework, we resort to a similar strategy. Since we are now considering theories that only interpret some arithmetical theory, we will consider the reflection principle
 \[
    (\forall x:N) (\mathtt{Prov}_\tau(\corn{\vphi(\dot{x})})\ra \vphi(x))
    \tag{$\mathtt{RFN}^N(\tau)$}
 \]
 $\mathtt{RFN}^N(\tau)$ is a schema, which ranges over all formulae $\vphi(v)$ of $\mc{L}_\tau$, but the initial universal quantifier is restricted to objects in the interpretation $N\colon \sot \to \tau$. 

 We first show that \textsc{generalized reflection} and \textsc{uniform $\nat$-invariance} have nontrivial logical strength. 
 \begin{proposition}\label{prop:rfnnin}
$\mc{I}(\tau)$ includes $\mathtt{RFN}^N(\tau)$.
 \end{proposition}
 \begin{proof}\hfill
 We consider again the axiomatization $\tau'$ of $\tau$ by reflection, although this time we mean $\mathtt{RFN}^N(\tau)$. With the contextual information that $x:N$, 
    \[
        \tau'(x):\lra \tau(x)\vee (\exists y\leq^N x)\;x=^N\corn{\mathtt{Proof}_\tau(y_1,\vphi(y_2))\ra \vphi(y_2)}
    \]
\noindent By \textsc{generalized reflection}, $\mathtt{RFN}^N(\tau)\in \mc{I}(\tau')$. The identity interpretations, and the provability of the ``small reflection principle'' \eqref{eq:smaref} witness that $\tau'\btrl_r \tau$.\footnote{We are  omitting here reference to the identity interpretations.}  
 By \textsc{uniform $\nat$-invariance}, the same interpretation(s) witnesses that $\mc{I}(\tau')\vtl_r \mc{I}(\tau)$, so $\mathtt{RFN}^N(\tau)\in \mc{I}(\tau)$. 
 \end{proof}

 We now proceed with the natural lower bound for the strength of \textsc{generalized reflection} and \textsc{uniform $\nat$-invariance}
 
 
 \begin{proposition}\label{prop:natint}
    \textsc{Generalized reflection} and \textsc{uniform $\nat$-invariance} are satisfied by
    \[
        \mc{I}_{\mathtt{RFN}^N}(\tau)=\{\vphi \in \mc{L}_\tau \sth \tau+\mathtt{RFN}^N(\tau)\}.
    \]
 \end{proposition}

 \begin{proof}
    That $\mc{I}_\mathtt{RFN}$ is satisfied by \textsc{Generalized reflection} follows immediately, since obviously $\tau$ proves that $\forall x : N\; \tau(\corn{\vphi(\dot{x})})$ only if $\forall x :N \;\mathtt{Prov}_\tau(\corn{\vphi(\dot{x})})$.

    For \textsc{uniform $\nat$-invariance}, given the assumption $\tau'\btrl^{F,G}_r \tau$ and the nature of $F,G$, it suffices to check that $F,G$ are indeed ($N$-preserving) interpretations between $\mc{I}_{\mathtt{RFN}^N}(\tau)$ and $\mc{I}_{\mathtt{RFN}^N}(\tau')$. We verify only the case of $F$, since the argument for $G$ is symmetric. Reasoning in $\mc{I}_{\mathtt{RFN}^N}(\tau)$, we show the translation of an arbitrary instance of $\mathtt{RFN}^N(\tau')$. We assume, for $x:N$, $\mathtt{Prov}^F_{\tau'}(\corn{\vphi(\dot{x})})$, that is $\mathtt{Prov}_{\tau'}(\corn{\vphi(\dot{x})})$ by $N$-preservation; since $\tau$ interprets $\sot$ via $N$ and 
    \[
        \sot \vdash \forall x(\mathtt{Prov}_{\tau'}(\corn{\vphi(\dot{x})})\ra \mathtt{Prov}_{\tau}(\corn{\vphi^F(\dot{x})})),
    \]
    we have $\mathtt{Prov}_{\tau}(\corn{\vphi^F(\dot{x})})$. Then $\mathtt{RFN}^N(\tau)$ gives us the claim. 
 \end{proof}
 \begin{remark}
    Continuing on the theme introduced in Remark \ref{rem:smooth1}, Propositions \ref{prop:rfnnin} and \ref{prop:natint} still hold if we require $\sot$-verifiable retracts on both sides of \textsc{uniform $\nat$-invariance}. In fact, Proposition \ref{prop:rfnnin} can be verified in $\sot$. For Proposition \ref{prop:natint}, one problem is that in weak metatheories such as $\sot$ the $\tau$-provability of the translations of \emph{axioms} of $\tau'$ (axiom interpretability) may not entail the $\tau$-provability of the translations of \emph{theorems} of $\tau'$ (theorem interpretability) \cite[\S5]{vis91}. It is known that for a special kind of interpretations, \emph{smooth interpretations} axiom- and theorem-interpretability are equivalent (see again \cite{vis91}). Back to Proposition \ref{prop:natint}, the  assumption $\tau'\btrl^{F,G}_r \tau$ witnesses that $F,G$ are {smooth} as interpretations of $\mc{I}_{\mathtt{RFN}^N}(\tau')$ and $\mc{I}_{\mathtt{RFN}^N}(\tau)$ (provably in $\sot$). This guarantees that the axiom-interpretability between $\mc{I}_{\mathtt{RFN}^N}(\tau')$ and $\mc{I}_{\mathtt{RFN}^N}(\tau)$ provided in our proof transforms into theorem-interpretability. 
 \end{remark}

 \section{Further Work}

We conclude the paper with some questions and issues left open by our study. 

\begin{itemize}
    \item What is the exact strength of $\tau_\Gamma$ as defined in section \ref{sec:iterat}?
    \item A study of the notion of an autonomous progression of an implicit commitment operator and a comparison with the progression studied in section \ref{sec:iterat}.
    \item Is it possible to abstract away the role of $\sot$ in the notion of \emph{adequate interpretation} (cf.~Def.~\ref{dfn:adeint}), and reason in terms of an abstract syntactic structure?\footnote{We are grateful to Bartosz Wcis\l o for suggesting this.}
    \item  An investigation of the formalization and iteration of generalized implicit commitments from section \ref{sec:genict}, both simple and autonomous. 
\end{itemize}

Also a full philosophical assessment of the generalized picture is also required; a particularly pressing question concerns the role of notions of theoretical equivalence in the individuation of mathematical commitments.

\section*{Acknowledgments}
We thank Bartosz Wcis\l o for useful remarks and an anonymous referee for detailed and helpful comments. Mateusz \L{}e\l{}yk's research was supported by an NCN Maestro grant 2019/34/A/HS1/00399, “Epistemic and semantic commitments of foundational theories.” Carlo Nicolai's research was supported by the AHRC Research Grant H/V015516/1 Properties, Paradox, and Circularity. A New, Type-Free Account'. 


\bibliography{IC_bib}
\bibliographystyle{plain}

\end{document}